\documentclass[12pt,english,BCOR7.5mm]{amsart}
\usepackage{ae,aecompl}
\usepackage[T1]{fontenc}
\usepackage[latin9]{inputenc}
\usepackage[letterpaper]{geometry}
\usepackage{color}
\usepackage{babel}
\usepackage{amsthm}
\usepackage{amstext}
\usepackage{amssymb}
\usepackage[unicode=true,
 bookmarks=true,bookmarksnumbered=true,bookmarksopen=true,bookmarksopenlevel=2,
 breaklinks=false,pdfborder={0 0 1},backref=false,colorlinks=true]
 {hyperref}
\hypersetup{pdftitle={Using XY-pc in LyX},
 pdfauthor={H. Peter Gumm},
 pdfsubject={LyX's XY-pic manual},
 pdfkeywords={LyX, documentation},
 linkcolor=black,citecolor=black,urlcolor=blue,filecolor=blue,pdfpagelayout=OneColumn,pdfnewwindow=true,pdfstartview=XYZ,plainpages=false,pdfpagelabels}

\makeatletter
\numberwithin{equation}{section}
\numberwithin{figure}{section}
\theoremstyle{plain}
\newtheorem{thm}{\protect\theoremname}
  \theoremstyle{plain}
  \newtheorem{conjecture}[thm]{\protect\conjecturename}
  \theoremstyle{plain}
  \newtheorem{cor}[thm]{\protect\corollaryname}
  \theoremstyle{plain}
  \newtheorem{lem}[thm]{\protect\lemmaname}
  \theoremstyle{definition}
  \newtheorem*{problem*}{\protect\problemname}
  \theoremstyle{plain}
  \newtheorem{prop}[thm]{\protect\propositionname}

\usepackage{babel}
\usepackage[all]{xy}

\newcommand{\xyR}[1]{
  \xydef@\xymatrixrowsep@{#1}}
\newcommand{\xyC}[1]{
  \xydef@\xymatrixcolsep@{#1}}

\newdir{|>}{!/4.5pt/@{|}*:(1,-.2)@^{>}*:(1,+.2)@_{>}}

\let\myTOC\tableofcontents
\renewcommand{\tableofcontents}{%
  \pdfbookmark[1]{\contentsname}{}
  \myTOC }

\def\LyX{\texorpdfstring{%
  L\kern-.1667em\lower.25em\hbox{Y}\kern-.125emX\@}
  {LyX}}

  \providecommand{\conjecturename}{Conjecture}
  \providecommand{\corollaryname}{Corollary}
  \providecommand{\lemmaname}{Lemma}
  \providecommand{\problemname}{Problem}
  \providecommand{\propositionname}{Proposition}
\providecommand{\theoremname}{Theorem}

\makeatother

  \providecommand{\conjecturename}{Conjecture}
  \providecommand{\corollaryname}{Corollary}
  \providecommand{\lemmaname}{Lemma}
  \providecommand{\problemname}{Problem}
  \providecommand{\propositionname}{Proposition}
\providecommand{\theoremname}{Theorem}

\begin{document}

\title{Zero-sum Subsequences of Length $kq$ over Finite Abelian $p$-Groups}

\author{Xiaoyu He}

\email{xiaoyuhe@college.harvard.edu}

\address{Eliot House, Harvard College, Cambridge, MA 02138.}
\begin{abstract}
For a finite abelian group $G$ and a positive integer $k$, let $s_{k}(G)$
denote the smallest integer $\ell\in\mathbb{N}$ such that any sequence
$S$ of elements of $G$ of length $|S|\geq\ell$ has a zero-sum subsequence
with length $k$. The celebrated Erd\H{o}s-Ginzburg-Ziv theorem determines
$s_{n}(C_{n})=2n-1$ for cyclic groups $C_{n}$, while Reiher showed
in 2007 that $s_{n}(C_{n}^{2})=4n-3$. In this paper we prove for
a $p$-group $G$ with exponent $\exp(G)=q$ the upper bound $s_{kq}(G)\le(k+2d-2)q+3D(G)-3$
whenever $k\geq d$, where $d=\Big\lceil\frac{D(G)}{q}\Big\rceil$
and $p$ is a prime satisfying $p\ge2d+3\Big\lceil\frac{D(G)}{2q}\Big\rceil-3$,
where $D(G)$ is the Davenport constant of the finite abelian group
$G$. This is the correct order of growth in both $k$ and $d$. As
a corollary, we show $s_{kq}(C_{q}^{d})=(k+d)q-d$ whenever $k\geq p+d$
and $2p\geq7d-3$, resolving a case of the conjecture of Gao, Han,
Peng, and Sun that $s_{k\exp(G)}(G)=k\exp(G)+D(G)-1$ whenever $k\exp(G)\geq D(G)$.
We also obtain a general bound $s_{kn}(C_{n}^{d})\leq9kn$ for $n$
with large prime factors and $k$ sufficiently large. Our methods
are inspired by the algebraic method of Kubertin, who proved that
$s_{kq}(C_{q}^{d})\leq(k+Cd^{2})q-d$ whenever $k\geq d$ and $q$
is a prime power. 
\end{abstract}
\maketitle

\section*{Introduction}

In 1961, Erd\H{o}s, Ginzburg and Ziv proved the following result,
sparking interest in the additive theory of sequences over finite
abelian groups. 
\begin{thm}
\label{thm:EGZ}\cite{EGZ} Let $n$ be an arbitrary positive integer.
Given any sequence $S$ of $2n-1$ integers, there is a subsequence
$T$ of $S$ with length $n$, the sum of whose terms is divisible
by $n$. 
\end{thm}
The natural generalization of this result is the study of sequences
over finite abelian groups which are guaranteed to have zero-sum subsequences
of some prescribed length.

We use the notation $[a,b]$ to denote the set of all positive integers
$\{a,a+1,\ldots,b\}$ between $a$ and $b$ inclusive.

Let $(G,+)$ be a finite abelian group written additively. Then, we
write $|G|$ for the size of $G$ and $\exp(G)$ for the exponent
of $G$, i.e. the largest order of any element of $G$. A sequence
$S$ over $G$ will be written multiplicatively in the form 
\[
S=g_{1}g_{2}\cdots g_{\ell}=\prod_{g\in G}g^{\mathsf{v}_{g}(S)},
\]
with $\mathsf{\mathsf{v}}_{g}(S)\geq0$ being the number of times
that $g$ appears in $S$.

With these definitions, we call 
\[
|S|=\sum_{g\in G}\mathsf{v}_{g}(S)
\]
the \emph{length} of $S$ and 
\[
\sigma(S)=\sum_{g\in G}\mathsf{v}_{g}(S)\cdot g
\]
the \emph{sum} of $S$ (which is an element of $G$). A sequence $S$
is zero-sum if $\sigma(S)=0$.

Throughout, we write $C_{m}$ for the cyclic group of order $m$ and
implicitly identify it with $\mathbb{Z}/m\mathbb{Z}$.

We say that $T$ is a \emph{subsequence }of $S$, written $T|S$,
if $\mathsf{v}_{g}(T)\leq\mathsf{v}_{g}(S)$ for every $g\in G$.

Following Gao and Thangadurai \cite{GT}, we define $s_{k}(G)$ to
be the smallest positive integer $\ell$ for which any sequence $S$
of length $\ell$ over $G$ contains a zero-sum subsequence of length
$k$. Usually we will only be concerned with the case where $\exp(G)|k$;
it is easy to check that if $\exp(G)\nmid k$ then $s_{k}(G)=\infty$.
Theorem \ref{thm:EGZ} proved that $s_{n}(C_{n})=2n-1$, where $C_{n}$
is the cyclic group of order $n$. The case $G=C_{n}^{d}$ and $k=n$
was first studied by Harborth \cite{Harborth}. 

It will henceforth be implicitly understood that tight lower bounds
on all of these quantities $s_{k}(G)$ can be proved by construction
and it suffices to prove tight upper bounds.

In the two-dimensional case it was first conjectured by Kemnitz \cite{Kemnitz}
that $s_{n}(C_{n}^{2})=4n-3$. Alon and Dubiner \cite{AlonDubiner,AlonDubiner2}
obtained the first linear bounds of the form $s_{n}(C_{n}^{2})\leq6n-5$.
Later Rónyai \cite{Ronyai} showed for primes $p$, $s_{p}(C_{p}^{2})\leq4p-2$,
and the full Kemnitz conjecture was resolved by Reiher \cite{Reiher}.
All of these results follow from algebraic considerations close to
the Chevalley-Warning theorem.

In this paper we are primarily interested in finite abelian $p$-groups,
for which the following conjecture has been made.
\begin{conjecture}
\label{mainconj}\cite{GT,Kubertin} Let $p$ be a prime and let $G$
be a finite abelian $p$-group with $\exp(G)=q$. Then for any $k\ge1$,
we have $s_{kq}(G)=kq+D(G)-1$.
\end{conjecture}
Here $D(G)$ denotes the Davenport constant of $G$, which is the
shortest length $\ell$ for which any sequence $S$ of length $\ell$
has some zero-sum subsequence. For a $p$-group of the form 
\[
G=\bigoplus_{i=1}^{d}C_{p^{\alpha_{i}}},
\]
the value was determined by Olson to be $D(G)=1+\sum(p^{\alpha_{i}}-1)$
\cite{Olson}. Conjecture \ref{mainconj} has been proved when $G$
has rank at most $2$, see for instance Theorem 6.13 in the survey
of Gao and Geroldinger \cite{GG}. 

It will be useful to define, using notation from Geroldinger, Grynkiewicz
and Schmid \cite{GGS}, for any set $K$ of positive integers the
value $s_{K}(G)$ to be the shortest length $\ell$ for which any
sequence $S$ of length $\ell$ over $G$ contains a zero-sum subsequence
with length in $K$. Geroldinger, Grynkiewicz and Schmid were interested
in the case that $K$ is an infinite arithmetic progression $a\mathbb{N}$,
but we will primarily work with finite sets $K$. Our first main result
gives a bound on $s_{K}(G)$ when $p$ is prime and $|K|\geq d$.
The same result was proved for $G=C_{q}^{d}$ by Kubertin \cite{Kubertin},
where $q$ is a power of $p$, and our argument is essentially identical.

Henceforth we write $d=\Big\lceil\frac{D(G)}{q}\Big\rceil$, which
we think of as the \emph{dimension }of $G$ as a $p$-group. Note
that in the case that $G=C_{q}^{d}$ and $q\ge d$ the two uses of
$d$ agree.
\begin{thm}
\label{thm:sets}Let $p$ be a prime and let $G$ be a finite abelian
$p$-group with $\exp(G)=q$ and $\Big\lceil\frac{D(G)}{q}\Big\rceil=d$.
If $K\subseteq[1,p]$ is a set satisfying $|K|\ge d$ , then 
\[
s_{Kq}(G)\le(\max K+1-|K|)q+D(G)-1,
\]
where $Kq=\{kq:k\in K\}$. 
\end{thm}
Gao, Han, Peng, and Sun \cite{Gao1,GHPS} have shown the following
preliminary results in the most general setting, assuming nothing
about $G$. 
\begin{thm}
\label{thm:gao}\cite{Gao1,GHPS} Let $G$ be a finite abelian group
with $\exp(G)=n$. Then for any $k\ge1$, we have $s_{kn}(G)\geq kn+D(G)-1$.
If $kn\geq|G|$, then equality holds, whereas if $kn<D(G)$, then
the inequality is strict. 
\end{thm}
From here they introduced the threshold function $\ell(G)$ defined
as the smallest posititve integer $\ell$ for which $s_{kn}(G)=kn+D(G)-1$
for any $k\geq\ell$. From Theorem \ref{thm:gao} we have 
\[
\frac{D(G)}{\exp(G)}\leq\ell(G)\leq\frac{|G|}{\exp(G)}.
\]

The lower bound is conjectured to be tight; our primary goals in this
paper are to bound the growth of $s_{k}(G)$ and in turn give a much
stronger upper bound on $\ell(G)$ for $p$-groups.

Using Theorem \ref{thm:sets} it is possible to prove the following
bound on $s_{kp}(G)$ when $G$ is a $p$-group. For comparison, Kubertin's
methods \cite{Kubertin} allow one to prove $s_{k}(C_{p}^{d})\leq(k+Cd^{2})p-d$
for some fixed constant $C>0$.
\begin{thm}
\label{thm:mainbound}Let $p$ be a prime, let $G$ be a finite abelian
$p$-group with $\exp(G)=q$ and $\Big\lceil\frac{D(G)}{q}\Big\rceil=d$.
If $p\ge2d+3\Big\lceil\frac{D(G)}{2q}\Big\rceil-3$, and $k\ge d$,
then
\[
s_{kq}(G)\le(k+2d-2)q+3D(G)-3.
\]

\end{thm}
When restricted to $G=C_{q}^{d}$, the Davenport constant is just
$D(G)=qd-d+1$, and the bound reduces to $s_{kq}(C_{q}^{d})\le(k+5d-2)q-3d$
when $2p\ge7d-3$ and $k\ge d$, achieving a bound linear in $d$
where Kubertin proved only a quadratic one. The third part of Theorem
\ref{thm:mainbound}, together with Olson's calculation of $D(G)$,
gives a new bound for $\ell(G)$ for $G$ a $p$-group.
\begin{cor}
\label{cor:lbound2}Let $p$ be a prime and let $G$ be a finite abelian
$p$-group with $\exp(G)=q$ and $\Big\lceil\frac{D(G)}{q}\Big\rceil=d$.
If $p\ge2d-2+\Big\lceil\frac{2D(G)-2}{q}\Big\rceil$, then $\ell(G)\leq p+d$.
That is, $s_{kq}(G)=kq+D(G)-1$ whenever $k\ge p+d$.
\end{cor}
In the next section, we collect two well-known lemmas about the behavior
of $s_{k}(G)$, before proceeding to the proofs of Theorems \ref{thm:sets}
and \ref{thm:mainbound}. For more discussion of the implications
of Theorems \ref{thm:mainbound} and Corollary \ref{cor:lbound2}
on the general problem and on open questions, see the final section.

\section*{Preliminary Lemmas}

First, we show an easy sub-additivity result on $s_{k}(G)$ for general
$G$. Note that if $\exp(G)\nmid a$ or $\exp(G)\nmid b$ the following
lemma is vacuously true. 
\begin{lem}
\label{lem:subadditive}If $G$ is a finite abelian group and $a,b\in\mathbb{N}$,
then 
\[
s_{a+b}(G)\leq\max\{s_{a}(G)+b,s_{b}(G)\}.
\]
\end{lem}
\begin{proof}
Suppose that $S$ is a sequence over $G$ with $|S|\geq s_{a}(G)+b$
and $|S|\geq s_{b}(G)$. By the latter inequality, $S$ has a zero-sum
subsequence $S_{1}$ of length $b$. By the former inequality, $SS_{1}^{-1}$
has a zero-sum subsequence $S_{2}$ of length $a$. It follows that
any given $S$ with the stated length contains a zero-sum subsequence
$S_{1}S_{2}$ of length $a+b$, as desired. 
\end{proof}
Also we will need an easy special case of Conjecture \ref{mainconj},
which follows directly from a result of Geroldinger, Grynkiewicz and
Schmid \cite{GGS}. We include a quick proof for convenience, using
the result of Olson \cite{Olson} which determines the value of the
$D(G)$ for all $p$-groups. 
\begin{lem}
\label{lem:pcase}Let $p$ be a prime and let $G$ be a finite abelian
$p$-group with $\exp(G)=q$ and $\Big\lceil\frac{D(G)}{q}\Big\rceil=d$.
If $p\ge d$ , then $s_{kpq}(G)=kpq+D(G)-1$ for any integer $k\ge1$.\end{lem}
\begin{proof}
It suffices by Theorem \ref{thm:gao} and Lemma \ref{lem:subadditive}
to prove that $s_{pq}(G)\leq pq+D(G)-1$. Let $S$ be a sequence of
this latter length over $G$.

Let $1$ be a generator of $C_{pq}$. Construct a sequence $S'$ over
$G\oplus C_{pq}$ such that if $S=\prod_{g\in G}g^{\mathsf{v}_{g}(S)}$,
then $S'=\prod_{g\in G}(g,1)^{\mathsf{v}_{g}(S)}$. Then, a zero-sum
subsequence of $S'$ corresponds exactly to a zero-sum subsequence
of $S$ with length divisible by $pq$. Since $|S|=pq+D(G)-1<2pq$,
it follows that any such subsequence would have length exactly $pq$.

Thus, since $D(G\oplus C_{pq})=pq+D(G)-1$ by Olson's theorem \cite{Olson},
it follows that $S$ itself had a length $pq$ zero-sum subsequence,
as desired. 
\end{proof}
Combining Lemmas \ref{lem:subadditive} and \ref{lem:pcase}, it remains
to control the behavior of $s_{kq}(G)$ in the finite interval $k\in[d,p-1]$.

\section*{The Algebraic Method of Rónyai and Kubertin}

In this section we extend the algebraic method of Rónyai \cite{Ronyai},
who showed that $s_{2}(C_{p}^{2})\leq4p-2$ for all primes $p$, to
prove Theorem \ref{thm:sets}. Theorem \ref{thm:sets} was proved
for $C_{q}^{d}$ in the paper of Kubertin \cite{Kubertin}.

We require the following elementary result. It was proved for fields
by Rónyai \cite{Ronyai}, but we will also require the case $R=\mathbb{Z}$,
which is no additional work.
\begin{lem}
\label{lem:ronyai} Let $R$ be an integral domain and $m$ a positive
integer. Then the (multilinear) monomials $\prod_{i\in I}x_{i},I\subseteq[1,m]$,
constitute a basis for the free $R$-module $M$ of all functions
from $\{0,1\}^{m}$ to $R$. (Here $0$ and $1$ are viewed as elements
of $R$.)\end{lem}
\begin{proof}
Since the indicator function of a given point $\textbf{y}=(y_{1},y_{2},\ldots,y_{m})\in\{0,1\}^{m}$
can be written as 
\[
p(\textbf{x})=(-1)^{m}\prod_{i=1}^{m}(x_{i}-1+y_{i}),
\]
and such a polynomial can be expanded into a $R$-linear combination
of the given monomials $\{\prod_{i\in I}x_{i},I\subseteq[1,m]\}$,
these monomials certainly generate $M$. But $M$ has rank $2^{m}$
so this set of generators is also a basis, as desired.
\end{proof}
Using this lemma, we can prove Theorem \ref{thm:sets}. 
\begin{proof}
(of Theorem \ref{thm:sets}) Let $G=C_{q_{1}}\oplus\cdots\oplus C_{q_{e}}$
with each $q_{i}=p^{m_{i}}$ for some $m_{i}>0$. Write $q=\exp(G)=\max(q_{i})$.
Let $S=\prod_{i=1}^{m}g_{i}$ be a sequence over $G$ with length
\begin{eqnarray*}
m & = & (\max K+1-|K|)q+D(G)-1\\
 & = & (\max K+1-|K|)q+\sum_{i=1}^{e}q_{i}-e.
\end{eqnarray*}
We show that, given any set $K$ of positive integers in $[1,p]$
with cardinality at least $d$, some zero-sum subsequence of $S$
has length in $Kq$. Suppose otherwise.

Working over the field $\mathbb{Q}$, we define the following polynomial
on $m$ variables $x_{1},x_{2},\ldots,x_{m}$. Write $\textbf{x}$
for the vector of all the $x_{i}$. By identifying $C_{q_{j}}$with
$\mathbb{Z}/q_{j}\mathbb{Z}$ and picking representatives $[0,q_{j}-1]\in\mathbb{Z}$
for this quotient, we let $a_{i}^{(j)}$ denote the representative
in $[0,q_{j}-1]$ for the $j$-th component of the $g_{i}\in S$,
where $i\in[1,m]$ and $j\in[1,e]$. Also, given a polynomial $Q(\textbf{x})\in\mathbb{Q}[\textbf{x}]$
and an integer $n\ge0$, we define 
\[
\binom{Q(\textbf{x})}{n}=\frac{Q(\textbf{x})(Q(\textbf{x})-1)\cdots(Q(\textbf{x})-(m-1))}{m!}\in\mathbb{Q}[\textbf{x}],
\]
with an empty product taken to be $1$. Then, define $P$ to be the
integer-valued polynomial

\texttt{
\[
P(\textbf{x})=P_{L}(\textbf{x})P_{S}(\textbf{x})P_{K}(\textbf{x}),
\]
}where 
\begin{eqnarray*}
P_{L}(\textbf{x}) & = & \binom{\sum_{i=1}^{m}x_{i}-1}{q-1}\\
P_{S}(\textbf{x}) & = & \prod_{j=1}^{e}\binom{\sum_{i=1}^{m}a_{i}^{(j)}x_{i}-1}{q_{j}-1}\\
P_{K}(\textbf{x}) & = & \prod_{\ell\in[1,\max K]\backslash K}\Big(\binom{\sum_{i=1}^{m}x_{i}}{q}-\ell\Big).
\end{eqnarray*}

If $\textbf{x}\in\{0,1\}^{m}$, then it uniquely indexes a subsequence
$T|S$ with the terms of $T$ precisely those $g_{i}$ for which $x_{i}=1$.
Now we show that $P$ vanishes except when $\textbf{x}=\textbf{0}$
(the null vector). Suppose $\textbf{x}\ne\textbf{0}$ and indexes
a sequence $T|S$. We repeatedly apply the classical result of Lucas,
which implies that for any power $q_{j}$ of $p$, 
\[
\binom{y-1}{q_{j}-1}\equiv\begin{cases}
1\pmod p & \mbox{if }q_{j}|y\\
0\pmod p & \mbox{otherwise},
\end{cases}
\]
and furthermore 
\[
\binom{\ell q_{j}}{q_{j}}\equiv\ell\pmod p,
\]
for any integers $y,\ell$. Thus the polynomial $P_{L}$ vanishes
modulo $p$ whenever $|T|$ is not a multiple of $q$, and the polynomial
$P_{S}$ vanishes modulo $p$ whenever $\sigma(T)\neq0$. Given that
$|T|$ is a multiple of $p$, the polynomial $P_{K}$ vanishes whenever
$|T|$ is exactly congruent to $\ell p\mod p^{2}$ where $\ell\in[1,\max K]\backslash K$.
It follows that the only possibility for $Q(\textbf{x})$ not to vanish
modulo $p$ is if $T$ is a zero-sum sequence with length congruent
to $\ell q$ for some $\ell\in K\cup[\max K+1,p]\pmod p$. But since
$|T|\le|S|=m$ and $m$ is constructed to be less than $(\max K+1)q$,
it follows that $|T|\in Kq$, contradiction.

We see that $P$ vanishes modulo $p$ on all vectors $\textbf{x}$
with the sole exception of the all-$0$'s vector. On that vector note
that according to Lucas' Theorem none of $P_{L},P_{S},P_{K}$ is zero.
Thus since $P$ is integer-valued, Lemma \ref{lem:ronyai} with $R=\mathbb{Z}/p\mathbb{Z}$
proves 
\[
P(\textbf{x})\equiv Q(\textbf{x})=C\prod_{i=1}^{m}(1-x_{i})+pQ_{1}(X)
\]
for some nonzero $C\in\mathbb{Z}$, not divisible by $p$, and some
integer-valued function $Q_{1}(X)$, as functions on $\{0,1\}^{m}$.
Furthermore, since $Q_{1}$ is integer-valued it is equal as a function
to some integer linear combination of monomials as in Lemma \ref{lem:ronyai}
with $R=\mathbb{Z}$, so we may as well take $Q_{1}\in\mathbb{Z}[\textbf{x}]$.
Finally since $C\not\equiv0\pmod p$ the top-degree term $\prod_{i\le m}x_{i}$
in $Q(\textbf{x})$ has a nonzero coefficient, so $\deg Q\ge m$. 

On the other hand, $P$ can be written as a linear combination of
basis monomials over $\mathbb{Q}$ in another way, simply by expanding
the product $P=P_{L}P_{S}P_{K}$ and applying the relation $x_{i}^{2}=x_{i}$
for functions on $\{0,1\}^{m}$. Both expansions represent $P$ in
terms of the basis from Lemma \ref{lem:ronyai} over $\mathbb{Q}$.
For these expressions to be equal, the degrees must equate; on the
other hand the second expression has degree at most $\deg P_{L}+\deg P_{S}+\deg P_{K}$.
When we compute this expression, we get 
\[
\deg P_{L}+\deg P_{S}+\deg P_{K}=(q-1)+(D(G)-1)+(\max K-|K|)q=m-1
\]
by the definition of $m$. This cannot agree with the degree of $Q$,
so we have a contradiction and the theorem is proved.
\end{proof}

\section*{Bounds on Small Lengths}

We now prove Theorem \ref{thm:mainbound}. We begin, of course, with
Theorem \ref{thm:sets} which gives us 
\[
s_{Kq}(G)\leq(\max K+1-|K|)q+D(G)-1
\]
whenever $|K|\geq d$ and $\max(K)\leq p$. As a first step, we obtain
a bound for $s_{K}(C_{p}^{d})$ when $|K|\geq d/2$, allowing for
$K$ half as large. 
\begin{lem}
\label{lem:half}Let $p$ be a prime and let $G$ be a finite abelian
$p$-group with $\exp(G)=q$ and $\Big\lceil\frac{D(G)}{q}\Big\rceil=d$.
If $K\subset\mathbb{N}$ is a finite set with $|K|\ge d/2$ and $2\max K+|K|\le p$,
then 
\[
s_{Kq}(G)\le(2\max K+1-|K|)q+D(G)-1.
\]
\end{lem}
\begin{proof}
For any zero-sum sequence $T$ with $|T|=nq$ and $2\max K+1\le n\le p+1$,
we can define $L=K\cup(n-K)\subseteq[1,p]$ having $|L|=2|K|\ge d$
and $\max L\le n-1$. Since
\[
|T|=nq\ge(n-2|K|)q+D(G)\ge(\max L+1-|L|)q+D(G)-1,
\]
we can apply Theorem \ref{thm:sets} to $T$ with length set $L$.
Thus $T$ has a zero-sum subsequence with length in $Lq=Kq\cup(n-K)q$.
However, if it had a zero-sum subsequence $T_{1}$ with length $(n-k)q$
and $k\in K$, then $TT_{1}^{-1}$ has length $kq$ with $k\in K$,
and is also zero-sum since $T$ itself is zero-sum. It follows that
$T$ has a zero-sum subsequence with length in $Kq$.

Let $S$ be a sequence over $G$ of length $(2\max K+1-|K|)q+D(G)-1$.
Now, let $K'=K\cup\{2\max(K)+i:i\in[1,|K|]\}$. We have $|K'|=2K\geq d$
and $\max(K')=2\max(K)+|K|\le p$, by hypothesis. Also, 
\[
|S|=(2\max K+1-|K|)q+D(G)-1\ge(\max K'+1-|K'|)q+D(G)-1,
\]
so by Theorem \ref{thm:sets} again, this time applied to $K'$ and
$S$, we see that any sequence $S$ satisfying $|S|\geq(2\max(K)+|K|+1)q+D(G)-1$
has a zero-sum subsequence $T$ in with length in $K'q$. If $|T|\in Kq$
we're done. Otherwise, $|T|=nq$ with 
\[
2\max K+1\le n\le2\max K+|K|\le p.
\]
But then $T$ has a zero-sum subsequence with length in $Kq$ by the
previous argument, and so $S$ does as well. 
\end{proof}
Next we prove a much stronger bound than Theorem \ref{thm:mainbound}
on the interval $k\in[2d-1,p]$.
\begin{lem}
\label{lem:2d}Let $p$ be a prime, let $G$ be a finite abelian $p$-group
with $\exp(G)=q$, and let $k\in[2d-1,p]$ be an integer. Then,
\[
s_{kq}(G)\le kq+2D(G)-2.
\]
\end{lem}
\begin{proof}
Let $S$ be a sequence over $G$ satisfying $|S|=kq+2D(G)-2$. Factor
$S=S_{1}S_{2}$ where $|S_{1}|=(k+1-d)q+D(G)-1$ and $|S_{2}|=(d-1)q+D(G)-1$. 

If $d=1$ then $G$ is cyclic and the result is a trivial consequence
of the Erd\H{o}s-Ginzburg-Ziv Theorem, so assume $d\ge2$. Let $K$
be any $d$-subset of $[1,2d-2]q$, and apply Theorem \ref{thm:sets}
to $S_{2}$ with length set $K$. By ranging $K$ through all possible
such subsets, we see that at least $d-1$ of the lengths in $[1,2d-2]q$
appear as the lengths of zero-sum subsequences of $S_{2}$. Together
with the empty subsequence, these lengths form a cardinality $d$
subset $L\subset[0,2d-2]q$ such that every length in $L$ is the
length of some zero-sum subsequence $T_{2}|S_{2}$.

It remains to show that some length in $kq-L$ is the length of a
zero-sum subsequence $T_{1}|S_{1}$. But $kq-L$ has cardinality $d$
and maximum at most $kq$. Since $|S_{1}|\ge(k+1-d)q+D(G)-1$ we can
apply Theorem \ref{thm:sets} to conclude that $S_{1}$ indeed contains
a zero-sum subsequence $T_{1}$ with length in $kq-L$. Combining
$T_{1}$ and $T_{2}$ we find that $S$ has a zero-sum subsequence
with the desired length $kq$. 
\end{proof}
Using Lemma \ref{lem:half} and Lemma \ref{lem:2d} together we can
prove Theorem \ref{thm:mainbound}. 
\begin{proof}
(of Theorem \ref{thm:mainbound}) For $k\in[2d-1,p]$, the result
follows by Lemma \eqref{lem:2d}, and since $p\ge2d-1$ this interval
is nonempty. 

Now suppose $d\ge2$ and $k\in[d,2d-2]$, so that $k\le p-1$. Let
$m=\lceil\frac{D(G)}{2q}\rceil$ and
\[
t=\Big\lfloor\frac{k}{2}\Big\rfloor+m-1.
\]

Note that since $k\ge d$ we have $t\ge2m-2$. Factor $S=S_{1}S_{2}$
with lengths satisfying
\begin{eqnarray*}
|S_{1}| & \ge & (2t-m+1)q+D(G)-1\\
|S_{2}| & \ge & (2k-2t+3m-3)q+D(G)-1.
\end{eqnarray*}

First, assume $m\ge2$ and $t>2m-2$. We can apply Lemma \ref{lem:half}
to $S_{1}$ with with all possible $m$-subsets $K$ of $[t-2m+2,t]$.
This is possible because for such a set $K$ we have $2\max K+|K|\le2t+m\le p$
by the hypothesis of the theorem. Thus there is a set $L\subset[t-2m+2,t]$
of cardinality $m$ such that every element of $Lq$ appears as the
length of some zero-sum sequence $T_{1}|S_{1}$. 

In the case that $t=2m-2$ exactly, we modify the argument slightly
by finding, along the same lines, an $L'\subset[t-2m+3,t]$ of cardinality
$m-1$ with this property, and then adding in the zero sequence to
form $L$. 

Finally in the case $m=1$ we have $t=0$ and we just take $ $$ $$L=\{0\}$,
and the desired properties still hold.

Now, we simply apply Lemma \ref{lem:half} to $S_{2}$ with the set
of lengths $kq-Lq$. This set has cardinality $m$ and maximum value
at most $k-t+2m-2$, and $p$ satisfies
\[
2(k-t+2m-2)+m\le2d+3m-3\le p,
\]
the conditions are satisfied and some $T_{2}|S_{2}$ has sum zero
and length $|T_{2}|\in kq-Lq$. Concatenating it with the corresponding
subsequence of $S_{1}$ the theorem is proved for $k\in[d,2d-2]$.

Finally it is easy to apply Lemma \ref{lem:subadditive} to prove
the theorem inductively on all $k\ge2d$. If $k\ge2d$, then we can
write $k=d+k'$ wih $k'\ge d$ and Lemma \ref{lem:subadditive} gives
\begin{eqnarray*}
s_{kq}(G) & \le & \max\{s_{dq}(G)+k'q,s_{k'q}(G)\}\\
 & \le & (k+2d-2)q+3D(G)-3
\end{eqnarray*}
and by induction the bound is proved for all $k\ge2d$.
\end{proof}
We briefly complete the proof of Corollary \ref{cor:lbound2}.
\begin{proof}
(of Corollary \ref{cor:lbound2}.) From Lemma \ref{lem:pcase}, we
have already $s_{pq}(G)=pq+D(G)-1$, and from Theorem \ref{thm:mainbound},
we have $s_{kq}(G)\le(k+2d-2)q+3D(G)-3$ if $k\ge d$. Combining these
via Lemma \ref{lem:subadditive}, we get
\begin{eqnarray*}
s_{kq}(G) & \le & \max\{s_{pq}(G)+(k-p)q,s_{(k-p)q}(G)\}\\
 & = & \max\{kq+D(G)-1,(k-p+2d-2)q+3D(G)-3\}.
\end{eqnarray*}

It suffices to show that under the assumptions of Corollary \ref{cor:lbound2},
the first term is the maximum. In fact, we are given
\begin{eqnarray*}
p & \ge & 2d-2+\Big\lceil\frac{2D(G)-2}{q}\Big\rceil\\
pq & \ge & (2d-2)q+2D(G)-2\\
kq+D(G)-1 & \ge & (k-p+2d-2)q+3D(G)-3,
\end{eqnarray*}
as desired.
\end{proof}

\section*{Closing Remarks and Open Problems}

We first make a few observations regarding the problem of Gao on the
threshold $\ell(G)$ after which $s_{kn}(G)=D(G)+kn-1$ for all $k\geq\ell(G)$,
where $n=\exp(G)$. Gao et al. \cite{GHPS} proved Theorem \ref{thm:gao}
which shows in general that 
\[
D(G)\leq n\ell(G)\leq|G|.
\]

It is conjectured by Gao et al. \cite{GHPS} that the lower bound
is tight, i.e. 
\[
\ell(G)=\bigg\lceil\frac{D(G)}{n}\bigg\rceil
\]
for all $G$. Thus we are mainly interested in improving the upper
bound $|G|/n$. In the case that $G$ is a $p$-group we can do much
better than $\ell(G)\leq|G|/n$ using Theorem \ref{thm:mainbound},
getting $\ell(G)\leq p+d$ when $p$ and $d$ satisfy the conditions
of Theorem \ref{thm:mainbound}.

For comparison, the results of Kubertin give $\ell(C_{p}^{d})\leq p+Cd^{2}$
for a constant $C>0$, while conjectural value is $\ell(C_{p}^{d})=d$,
so any bound independent of $p$ would be a significant improvement
on Theorem \ref{cor:lbound2}. The only available method for proving
bounds on $\ell(G)$ is combining bounds of the sort in Theorem \ref{thm:mainbound}
with Lemma \ref{lem:pcase}, which depends on $p$. 
\begin{problem*}
Can we remove the dependence on $p$ in Corollary \ref{cor:lbound2}? 
\end{problem*}
Just as in the special case $k=1$, bounds on $s_{kp}(C_{p}^{d})$
give rise to bounds on $s_{k}(C_{n}^{d})$, although in general the
dependence is weaker. As an easy consequence of Theorems \ref{thm:mainbound}
and \ref{cor:lbound2} we prove the following multiplicativity lemma. 
\begin{lem}
\label{lem:inductivegeneral}Let $p$ be a prime, let $q$ be a power
of $p$, and let $G$ be a finite abelian group with $\exp(G)=qn$
such that the quotient group $H=G/qG$ satisfies $\Big\lceil\frac{D(G)}{q}\Big\rceil=d$
and $p\ge2d+3\Big\lceil\frac{D(G)}{2q}\Big\rceil-3$. If $a,b>0$,
then 
\[
s_{abqn}(G)\leq s_{an}(qG)bq+(2d-2)q+3D(G)-3,
\]
and if furthermore $p\ge2d-2+\Big\lceil\frac{2D(G)-2}{q}\Big\rceil$
and $b\ge p+d$, then 
\[
s_{abqn}(G)\leq s_{an}(qG)bq+D(G)-1.
\]
\end{lem}
\begin{proof}
Given any sequence $S$ over $G$ with length at least $s_{an}(qG)bq+(2d-2)q+3D(G)-3$,
we can repeatedly remove, using Theorem \ref{thm:mainbound} on $G/qG$,
length $bq$ subsequences of $S$ whose sums lie in $qG$ until the
length falls below $(b+2d-2)q+3D(G)-3$. This can be repeated to extract
a total of $s_{an}(qG)$ disjoint zero-sum subsequences. The same
can be done using Corollary \ref{cor:lbound2} instead if $b\geq p+d$.

In either case, we end up with $s_{an}(qG)$ disjoint subsequences
of $G$, each of length $bq$ and having sum in $qG$. Thus there
is a zero-sum subsequence of the sequence of their sums, with length
$an$, corresponding to a zero-sum subsequence of $S$ with length
$abqn$ as desired.
\end{proof}
We can bound $s_{k}(G)$ directly from Lemma \ref{lem:inductivegeneral}
by induction. The empty product is taken to be $1$.
\begin{prop}
\label{prop:generalbound}Let $G$ be a finite abelian group with
$\exp(G)=n$, decomposed as
\[
G=\bigoplus_{i=1}^{r}G_{p_{i}},
\]
a direct sum of $p_{i}$-groups $G_{p_{i}}$ with $\exp(G_{p_{i}})=q_{i}$
and $\Big\lceil\frac{D(G_{p_{i}})}{q_{i}}\Big\rceil=d_{i}$, satisfying
$p_{i}\ge2d_{i}+3\Big\lceil\frac{D(G_{p_{i}})}{2q_{i}}\Big\rceil-3$.
Then, 
\[
s_{kn}(G)\leq kn+\sum_{i=0}^{r-1}\bigg(\prod_{j=1}^{i}a_{i}q_{i}\bigg)((2d_{i+1}-2)q_{i+1}+3D(G_{p_{i+1}})-3),
\]
where $k$ is any positive integer that can be written as a product
$k=a_{1}\cdots a_{r}$ of positive integers $a_{i}\ge d_{i}$. Also,

\[
s_{kn}(G)\le kn+\sum_{i=0}^{r-1}\bigg(\prod_{j=1}^{i}a_{i}q_{i}\bigg)(D(G_{p_{i+1}})-1)
\]
if each $p_{i}$ satisfies $p_{i}\ge2d_{i}-2+\Big\lceil\frac{2D(G_{p_{i}})-2}{q_{i}}\Big\rceil$
and each $a_{i}$ satisfies $a_{i}\ge p_{i}+d_{i}$.\end{prop}
\begin{proof}
Apply Lemma \ref{lem:inductivegeneral} with the filtration $G_{j}=\bigoplus_{i=1}^{j}G_{p_{i}},j=0,\ldots,r$
of $G$. Each subquotient $G_{j}/G_{j-1}\simeq G_{p_{j}}$ is a $p_{j}$-group
so the lemma applies.
\end{proof}
As a corollary, we have the following inequality by bounding the error
term crudely by a geometric series. For clarity, we state it in terms
of groups of the form $C_{n}^{d}$ though bounds on any finite abelian
group can be made in the same way.
\begin{cor}
\label{cor:generalbound}For $d>0$, $n=p_{1}p_{2}\cdots p_{r}$ with
not necessarily distinct prime factors $p_{1},\ldots,p_{r}\geq\frac{7}{2}d-3$,
and $k=a_{1}a_{2}\cdots a_{r}$ a product of positive integers $a_{1},a_{2},\ldots,a_{r}\geq d$,

\[
s_{kn}(C_{n}^{d})\leq9kn.
\]
If furthermore each $p_{i}$ satisfies $p_{i}\ge4d-2$ and each $a_{i}$
satisfies $a_{i}\ge p_{i}+d$, then
\[
s_{kn}(C_{n}^{d})\le3kn.
\]

\end{cor}
This is stronger than can be obtained by the iterative application
of Alon and Dubiner's general bounds \cite{AlonDubiner,AlonDubiner2}
on $s_{n}(C_{n}^{d})$, but only holds for a thin set of pairs $(k,n)$.
Of course, for any given $n$ satisfying the conditions of Proposition
\ref{prop:generalbound}, the inequality can be extended to all values
of $k$ in the semigroup generated additively by the $k$ satisfying
the stated condition, by Lemma \ref{lem:subadditive}, giving $s_{kn}(C_{n}^{d})\le9kn$
for all $k\ge d^{r}(d+1)^{r}$, where $r$ is the number of distinct
prime factors of $n$. Any technique achieving a bound of a strength
similar to that of Corollary \ref{cor:generalbound} but with the
threshold of $k$ independent of $n$ would be significant.

\section*{Acknowledgements}

This research was conducted at the Duluth Research Experience for
Undergraduates program in 2014, supported by National Science Foundation
grant number DMS-1358659 and National Security Agency grant number
H98230-13-1-0273. The author would like to thank Professor Joe Gallian
of the University of Minnesota-Duluth for organizing the program and
for his support in all stages of this work. Also, thanks go out to
Daniel Kriz for his many helpful suggestions. Input from an anonymous
referee was essential for reformulations of our results to their most
general form.


\begin{thebibliography}{10}
\bibitem[1]{AlonDubiner}N. Alon and M. Dubiner, \emph{A lattice point
problem and additive number theory,} Combinatorica 15-3 (1995), 301\textendash{}309.

\bibitem[2]{AlonDubiner2}N. Alon and M. Dubiner, \emph{Zero-sum sets
of prescribed size, }Combinatorics, Paul Erd\H{o}s Is Eighty, Vol.
1 (D. Miklos, V. T. Sos, and T. Szonyi, eds.), Bolyai Soc. Math.

\bibitem[3]{EGZ}P. Erd\H{o}s, A. Ginzburg, and A. Ziv, \emph{Theorem
in the additive number theory,} Bull. Research Council Israel 10F
(1961), 41\textendash{}43.

\bibitem[4]{Gao1}W. Gao, \emph{A combinatorial problem on finite
abelian groups}, J. Number Theor. 58 (1996), 100\textendash{}103.

\bibitem[5]{Gao2}W. Gao, \emph{On zero-sum subsequences of restricted
size II, }Discrete Math. 271 (2003), 51\textendash{}59.

\bibitem[6]{GG}W. Gao and A. Geroldinger, \emph{Zero-sum problems
in finite abelian groups: a survey}, Expo. Math. 24 (2006), 337\textendash{}369.

\bibitem[7]{GGS}A. Geroldinger, D. Grynkiewicz, and W. Schmid, \emph{Zero-sum
problems with congruence conditions}, Acta Math. Hungar., 131-4 (2011),
323\textendash{}345.

\bibitem[8]{GHPS}W. Gao, D. Han, J. Peng, and F. Sun, \emph{On zero-sum
subsequences of length $k\exp(G)$, }J. Combin. Theory Ser. A 125
(2014), 240\textendash{}253.

\bibitem[9]{GT}W. Gao and R. Thangadurai, \emph{On zero-sum sequences
of prescribed length}, Aequationes Math. 72 (2006), 201\textendash{}212.

\bibitem[10]{Harborth}H. Harborth, \emph{Ein Extremalproblem Für
Gitterpunkte J. Reine Angew. Math.}, 262/263 (1973), 356\textendash{}360.

\bibitem[11]{Kemnitz}A. Kemnitz, \emph{Extremalprobleme für Gitterpunkte},
Ph.D. Thesis, Technische Universität Braunschweig, 1982.

\bibitem[12]{Kubertin}S. Kubertin, \emph{Zero-sums of length $kq$
in $\mathbb{Z}_{q}^{d}$, }Acta Arith. 116-2 (2005), 145\textendash{}152.

\bibitem[13]{Olson}J. E. Olson, \emph{On a combinatorial problem
on finite Abelian groups I and II}, J. Number Theory 1 8-10 (1969),
195\textendash{}199.

\bibitem[14]{Reiher}C. Reiher, \emph{On Kemnitz' conjecture concerning
lattice-points in the plane}, Ramanujan J. 13 (2007), 333\textendash{}337.

\bibitem[15]{Ronyai}L. Rónyai, \emph{On a conjecture of Kemnitz},
Combinatorica 20-4 (2000), 569\textendash{}573.

\bibitem[16]{SavchezChen}S. Savchez and F. Chen, \emph{Long $n$-zero-free
sequences in finite cyclic groups, }Discrete Math. 308 (2008), 1\textendash{}8.\end{thebibliography}
\end{document}